\documentclass[12pt]{amsart}

\usepackage{amsmath}
\usepackage{amssymb}  
\usepackage{latexsym} 
\usepackage{comment}
\usepackage{url}
\usepackage{fullpage,url,amssymb,amsmath,amsthm,amsfonts,mathrsfs}
\usepackage[usenames,dvipsnames]{color}
\usepackage[pagebackref = true, colorlinks = true, linkcolor = blue, citecolor = Green]{hyperref}
\usepackage[alphabetic,lite]{amsrefs}
\usepackage{enumitem}
\usepackage{amscd}   
\usepackage[all, cmtip]{xy} 
\usepackage{xfrac}

\DeclareFontEncoding{OT2}{}{} 
\newcommand{\textcyr}[1]{%
 {\fontencoding{OT2}\fontfamily{wncyr}\fontseries{m}\fontshape{n}\selectfont #1}}

\usepackage[T2A,T1]{fontenc}
\makeatletter
\def\easycyrsymbol#1{\mathord{\mathchoice
  {\mbox{\fontsize\tf@size\z@\usefont{T2A}{\rmdefault}{m}{n}#1}}
  {\mbox{\fontsize\tf@size\z@\usefont{T2A}{\rmdefault}{m}{n}#1}}
  {\mbox{\fontsize\sf@size\z@\usefont{T2A}{\rmdefault}{m}{n}#1}}
  {\mbox{\fontsize\ssf@size\z@\usefont{T2A}{\rmdefault}{m}{n}#1}}
}}
\makeatother
\newcommand{\Bcyr}{\easycyrsymbol{\CYRB}}

\usepackage[all]{xy}
\usepackage{fullpage}
\newcommand{\Sha}{{\mbox{\textcyr{Sh}}}}

\usepackage{color} 

\newcommand{\defi}[1]{\textsf{#1}} 

\def\act#1#2%
  {\mathop{}%
   \mathopen{\vphantom{#2}}^{#1}%
   \kern-3\scriptspace%
   #2}

\newcommand{\Z}{{\mathbb Z}}
\newcommand{\Q}{{\mathbb Q}}

\newcommand{\A}{{\mathbb A}}

\newcommand{\G}{{\mathbb G}}
\newcommand{\Abar}{{\overline{A}}}

\newcommand{\Ebar}{{\overline{E}}}

\newcommand{\Kbar}{{\overline{K}}}

\newcommand{\Xbar}{{\overline{X}}}

\newcommand{\calL}{{\mathcal L}}
\newcommand{\calM}{{\mathcal M}}
\newcommand{\calN}{{\mathcal N}}

\newcommand{\To}{\longrightarrow}

\DeclareMathOperator{\Mat}{Mat}
\DeclareMathOperator{\inv}{inv}
\DeclareMathOperator{\CH}{CH}

\DeclareMathOperator{\Hom}{Hom}

\DeclareMathOperator{\ord}{ord}

\DeclareMathOperator{\Gal}{Gal}
\DeclareMathOperator{\Cor}{Cor}

\DeclareMathOperator{\Br}{Br}

\DeclareMathOperator{\Pic}{Pic}
\DeclareMathOperator{\Alb}{Alb}

\DeclareMathOperator{\HH}{H}
\DeclareMathOperator{\NS}{NS}

\DeclareMathOperator{\Spec}{Spec}

\DeclareMathOperator{\GL}{GL}

\newtheorem{Theorem}{Theorem}[]
\newtheorem{Lemma}[Theorem]{Lemma}
\newtheorem{Proposition}[Theorem]{Proposition}
\newtheorem{Corollary}[Theorem]{Corollary}

\newtheorem{Remark}[Theorem]{Remark}

\numberwithin{equation}{section}

\begin{document}

\title{There are no transcendental Brauer-Manin obstructions on abelian varieties}

\author{Brendan Creutz}
\address{School of Mathematics and Statistics, University of Canterbury, Private Bag 4800, Christchurch 8140, New Zealand}
\email{brendan.creutz@canterbury.ac.nz}
\urladdr{http://www.math.canterbury.ac.nz/\~{}b.creutz}

\begin{abstract}
	Suppose $X$ is a torsor under an abelian variety $A$ over a number field. We show that any adelic point of $X$ that is orthogonal to the algebraic Brauer group of $X$ is orthogonal to the whole Brauer group of $X$. We also show that if there is a Brauer-Manin obstruction to the existence of rational points on $X$, then there is already an obstruction coming from the locally constant Brauer classes. These results had previously been established under the assumption that $A$ has finite Tate-Shafarevich group. Our results are unconditional.
\end{abstract}

\maketitle
%
%
%
\section{Introduction}
%

	Let $X$ be a smooth projective and geometrically integral variety over a number field $k$. In order that $X$ possesses a $k$-rational point it is necessary that $X$ has points everywhere locally, i.e., that the set $X(\A_k)$ of adelic points on $X$ is nonempty. The converse to this statement is called the \defi{Hasse principle}, and it is known that this can fail. When $X(k)$ is nonempty one can ask if \defi{weak approximation} holds, i.e., if $X(k)$ is dense in $X(\A_k)$ in the adelic topology. 

	Manin \cite{Manin} showed that the failure of the Hasse principle or weak approximation can, in many cases, be explained by a reciprocity law on $X(\A_k)$ imposed by the Brauer group, $\Br X := \HH^2_\text{\'et}(X,\G_m)$. Specifically, each element $\alpha \in \Br X$ determines a continuous map, $\alpha_* \colon X(\A_k) \to \Q/\Z$, between the adelic and discrete topologies with the property that the subset $X(k)\subset X(\A_k)$ of rational points is mapped to $0$. Thus, for any subset $G \subset \Br X$, the set $X(\A_k)^G$ of adelic points that are mapped to $0$ by every element of $G$ is a closed subset of $X(\A_k)$ containing $X(k)$. When this subset is empty (resp., proper) one says there is a \defi{Brauer-Manin obstruction} to the existence of rational points (resp., weak approximation).

	The Brauer group of $X$ admits a natural filtration,
	\begin{equation}\label{filtration}
		0 \subset \Br_0X \subset \Bcyr X \subset \Br_{1/2}X \subset \Br_1X \subset \Br X\,,
	\end{equation}
	where the nontrivial subgroups, whose definitions are recalled in Section~\ref{notation} below, consist of elements that are  (respectively, in increasing order) \defi{constant}, \defi{locally constant}, \defi{Albanese}, and \defi{algebraic}. Elements of $\Br X$ surviving the quotient $\Br X/\Br_1X$ are said to be \defi{transcendental}. This filtration leads to the sequence of containments,
	\[
		X(\A_k)^{\Br} \subset X(\A_k)^{\Br_1} \subset X(\A_k)^{\Br_{1/2}}  \subset X(\A_k)^{\Bcyr} \subset X(\A_k)^{\Br_0} = X(\A_k)\,.
	\]
	In general, it is possible for any of these containments to be proper. When $X(\A_k)^{\Br} \ne X(\A_k)^{\Br_1}$ one says there is a \defi{transcendental Brauer-Manin obstruction}. We show that this cannot occur for torsors under abelian varieties.
	\vfill

\begin{Theorem}\label{thm:mainthm}
	Let $X$ be a torsor under an abelian variety over a number field $k$. Then
	\[
		X(\A_k)^{\Br} = X(\A_k)^{\Br_{1}} = X(\A_k)^{\Br_{1/2}}\,.
	\]
	Moreover, if these sets are empty then so is $X(\A_k)^{\Bcyr}$.
\end{Theorem}

The first statement implies that there are no transcendental obstructions to weak approximation on abelian varieties. The second statement says that the locally constant (\textit{a fortiori} algebraic) Brauer classes on $X$ capture the Brauer-Manin obstruction to the existence of rational points. This solves a problem stated in the ``American Institute of Mathematics Brauer groups and obstruction problems problem list'' \cite[Problem 2.4]{AimPL}.

We hasten to note that there are abelian varieties $A$ for which $\Br_{1/2}A \ne \Br A$. Examples with $\Br_1 A \ne \Br A$ can be found in \cite{SZ12}, while examples with $\Br_{1/2}A \ne \Br_1 A$ are described in Remark~\ref{rem:Brne} below.

The conclusion of Theorem~\ref{thm:mainthm} was previously known under the assumption that the Tate-Shafarevich group of the Albanese variety of $X$ is finite. When $X(k) \ne \emptyset$ this is due to Wang \cite{Wang}. When $X(k) = \emptyset$, this follows from Manin's result \cite[Th\'eor\`eme 6]{Manin} relating the obstruction coming from $\Bcyr X$ to the Cassels-Tate pairing.

Theorem~\ref{thm:mainthm} follows immediately from the next two theorems which are proved in Section~\ref{sec:proofs}.

\begin{Theorem}\label{prop1}
	If $X$ is a torsor under an abelian variety over a number field $k$, then $X(\A_k)^{\Br} = X(\A_k)^{\Br_{1/2}}$.
\end{Theorem}

	When $X = A$ is an abelian variety with finite Tate-Shafarevich group Theorem~\ref{prop1} follows from results of Wang \cite[Propositions 2 and 3]{Wang}. Her argument uses a result of Serre \cite{Serre} on the congruence subgroup problem to identify the profinite completion of $A(k)$ with the topological closure of $A(k)$ in the space $A(\A_k)_\bullet$ obtained from $A(\A_k)$ by replacing each factor $A(k_v)$ with its set of connected components. Descent theory together with finiteness of $\Sha(k,A)$ then imply that $\overline{A(k)} = A(\A_k)_\bullet^{\Br_{1/2}}$. As $A(\A_k)^{\Br}_\bullet$ is closed and contains the image of $A(k)$, this forces the equality in the statement of the theorem. Our proof of Theorem~\ref{prop1} is based on a recent result of the author and Viray in \cite{CV17}, and requires no assumption on $\Sha(k,A)$.

\begin{Theorem}\label{prop2}
	If $X$ is a torsor under an abelian variety over a number field $k$ such that $X(\A_k)^{\Bcyr} \ne \emptyset$, then $X(\A_k)^{\Br_{1/2}} \ne \emptyset$.
\end{Theorem}

This theorem follows from an argument using descent theory and \cite[Th\'eor\`eme 6]{Manin} and is likely to be known to the experts (see Proposition~\ref{prop:desc} for details). Theorem~\ref{prop2} is also an immediate consequence of the following result proved in Section~\ref{sec:proofs} by using compatibility of the Tate and Brauer-Manin pairings.

\begin{Theorem}\label{thm:completelycaptures}
	Suppose $X$ is a torsor under an abelian variety over a number field $k$ and $B \subset \Br_{1/2}X$ is a subgroup such that $X(\A_k)^B = \emptyset$. Then $B$ contains a locally constant class $\beta \in \Bcyr X$ such that $X(\A_k)^\beta = \emptyset$.
\end{Theorem}

By \cite[Th\'eor\`eme 6]{Manin}, finiteness of Tate-Shafarevich groups implies that the Brauer-Manin obstruction explains all failures of the Hasse principle for torsors under abelian varieties. Using Theorem~\ref{thm:mainthm} we deduce the converse.

\begin{Theorem}\label{cor1}
	The Brauer-Manin obstruction is the only obstruction to the existence of rational points on torsors under abelian varieties over number fields if and only if the maximal divisible subgroup of the Tate-Shafarevich group of every abelian variety over a number field is trivial.
\end{Theorem}

	After reading a draft of this paper, Olivier Wittenberg noted that Theorem~\ref{cor1} is stated without proof (and in a special case) in a recent paper of Harpaz \cite[page 3, line 5]{Harpaz}. In private communication with Wittenberg which predated the present work, Harpaz outlined a proof of Theorem~\ref{prop2} similar to the argument in Proposition~\ref{prop:desc} below and noted that Theorem~\ref{cor1} then follows from results in his joint work with Schlank concerning \'etale homotopy types~\cite[Corollary 1.2 or Theorem 12.1]{HarpazSchlank}. In fact, this approach can also be used to give a proof of Theorem~\ref{prop1} and, hence, Theorem~\ref{thm:mainthm}. Our approach avoids the heavy machinery of \'etale homotopy types and yields (in addition to Theorem~\ref{thm:completelycaptures}) the following more explicit result in the case $X(\A_k)^{\Br} = \emptyset$.
	
	\begin{Theorem}\label{cor2}
		Let $A$ be an abelian variety over a number field $k$. There exists and integer $p \ge 1$ such that if $X$ is a locally soluble torsor under $A$ and $B \subset \Br X[n]$ is a subgroup such that $X(\A_k)^{B}=\emptyset$. Then the class of $X$ in $\Sha(k,A)$ is not divisible by $n^p$. In particular, there is an element $\beta \in \Bcyr X[n^p]$ such that $X(\A_k)^\beta = \emptyset$ and some $Y$ in $\Sha(k,\hat{A})[n^p]$ pairing nontrivially with $X$ under the Cassels-Tate pairing on the Tate-Shafarevich groups of the Albanese and Picard varieties, $A = \Alb^0_X$ and $\hat{A} = \Pic^0_X$.
\end{Theorem}

\section{The filtration on $\Br X$}\label{notation}
	Here are the definitions of the subgroups appearing in the filtration~\eqref{filtration}. Other than $\Bcyr X$, all groups may be defined for a smooth projective integral variety $X$ over an arbitrary field $K$. Let $\Xbar$ denote the base change of $X$ to a separable closure of $K$. The \defi{algebraic} Brauer group, $\Br_1X$, is defined to be the kernel of the base change map $\Br X \to \Br \Xbar$. The group of \defi{constant} Brauer classes, $\Br_0X$, is defined to be the image of $\Br K := \Br \Spec(K)$ under the map induced by the structure morphism $X \to \Spec K$. The Hochschild-Serre spectral sequence gives rise to a well-known exact sequence (cf. \cite[(2.23) on p. 30]{Skorobogatov}),
	\begin{equation}\label{eq:definer}
		0 \to \Pic X \to \HH^0(K,\Pic \Xbar) \to \Br K \to \Br_1 X \stackrel{r}\to \HH^1(K,\Pic\Xbar) \to \HH^3(K,\G_m)\,.
	\end{equation}
There is an exact sequence
	\begin{equation}\label{eq:NSX}
		0 \to \Pic^0\Xbar \to \Pic \Xbar \to \NS\Xbar \to 0\,,
	\end{equation}
	where $\NS\Xbar$ is the N\'eron-Severi group of $\Xbar$. The inclusion $i\colon\Pic^0\Xbar \subset \Pic\Xbar$ induces a map $i_*\colon\HH^1(K,\Pic^0\Xbar) \to \HH^1(K,\Pic\Xbar)$. The group of \defi{Albanese} Brauer classes, $\Br_{1/2}X$, is defined to be $r^{-1}(i_*(\HH^1(K,\Pic^0\Xbar)))$, where $r$ is as in~\eqref{eq:definer}. Our terminology here comes from the fact that $\HH^1(K,\Pic^0\Xbar)$ parametrizes $K$-torsors under the dual of the Albanese variety of $X$. The notation $\Br_{1/2}$ was first introduced in~\cite{Stoll}.
	
	Now suppose $K = k$ is a number field. For a prime $v$ of $k$, let $X_{k_v}$ denote the base change of $X$ to the completion $k_v$ of $k$. We define the group of \defi{locally constant} Brauer classes, $\Bcyr X$, to be the subgroup of $\Br_{1/2}X$ consisting of those elements which map, under base change to $k_v$, into the subgroup $\Br_0X_{k_v} \subset \Br_{1/2}X_{k_v}$, for every prime $v$ of $k$.

\begin{Remark}
	We have defined $\Bcyr X$ as a subgroup of $\Br_{1/2}X$. A priori, this may be smaller than the group $\Bcyr X$ defined in \cite[p. 97]{Skorobogatov}, which is the subgroup of $\Br_1X$ mapping into $\Br_0X_{k_v}$ at all primes. One could also consider the subgroup of locally constant classes in $\Br X$. In general one does not expect these groups to be equal. Theorems~\ref{thm:mainthm} and~\ref{prop2} are strongest using our definition of $\Bcyr X$.
\end{Remark}

\begin{Remark}
	The group $B$ featuring in \cite[Th\'eor\`eme 6]{Manin} is defined on p. 405 of op. cit. to be $B :=  r^{-1}(i_*(\Sha(k,\Pic^0\Xbar)))$. A priori this is a subgroup of $\Bcyr X$. Corollary~\ref{cor:1} below shows that for $X$ a torsor under an abelian variety the map $i_*$ is injective over every completion $k_v$ of $k$ for which $X$ has $k_v$-points. From this we deduce that $B = \Bcyr X$ when $X$ is locally soluble. However, as we now show, this does not hold in general. For $X$ a genus $1$ curve, the kernel of $i_*$ is generated by the class of $X$ in $\HH^1(k,\Pic^0\Xbar) \simeq \HH^1(k,A)$. Suppose $X$ is not locally soluble at at least two places. If $A(k)$ and $\Sha(k,A)$ are both trivial, then there is an isomorphism $\HH^1(k,A) \simeq \bigoplus \HH^1(k_v,A)$ (See \cite[page 83]{MilneADT}). Hence there is some $Y \in \HH^1(k,A)$ that is locally soluble at all but one place $w$, where it is isomorphic to $X_{k_w}$. The image of $Y$ is a nontrivial element in $r(\Bcyr X)$, but $r(B) = 0$.
\end{Remark}	

\begin{Remark}\label{rem:Brne}
	For $k$ a number field the map $r \colon \Br_1 X \to \HH^1(k,\Pic\Xbar)$ is surjective, since $\HH^3(k,\G_m) = 0$ by \cite[Corollary I.4.21]{MilneADT}. Also, $\HH^2(k,\Pic^0\Xbar)$ is annihilated by $2$, since it is isomorphic to $\bigoplus_{v \text{ real}}\HH^2(k_v,\Pic^0\Xbar)$ by \cite[I.6.26(c)]{MilneADT}. It follows that any nontrivial element of odd order in $\HH^1(k,\NS\Xbar)$ is the image under $\Br_1 X \to \HH^1(k,\Pic\Xbar) \to \HH^1(k,\NS\Xbar)$ of a class that is not contained in $\Br_{1/2}X$. If $X = E \times E'$ is a product of elliptic curves, then $\HH^1(k,\NS\Xbar) \simeq \HH^1(k,\Hom(\Ebar,\Ebar'))$, yielding many examples of abelian varieties $X$ for which $\Br_1X \ne \Br_{1/2}X$.
\end{Remark}

\section{Five lemmas}

In this section $A$ is an abelian variety over a field $K$ of characteristic $0$ with algebraic closure $\Kbar$ and $X$ is an $A$-torsor. We use $\Xbar$ and $\Abar$ to denote the base change to $\Kbar$. The torsor structure is given by a morphism $\mu\colon A\times X \to X$. Any $a \in A(K)$ gives rise to translation maps $t_a \colon A \to A$ and $\tau_a \colon X \to X$ determined by $t_a(b) = a + b$ and $\tau_a(x) = \mu(a,x)$. For $x \in X(\Kbar)$ the torsor action gives an isomorphism $\tau_x \colon \Abar \to \Xbar$ defined by $\tau_x(a) = \tau_a(x) = \mu(a,x)$.

	There is a homomorphism $\HH^0(K,\NS\Abar) \to \Hom_{k}(A(\Kbar),\Pic^0\Abar)$ sending $\lambda \in \HH^0(K,\NS\Abar)$ to the morphism of Galois modules $\phi_\lambda\colon A(\Kbar) \to \Pic^0\Abar$ defined by $\phi_\lambda(a) = t_a^*\calN\otimes \calN^{-1}$, where $\calN \in \Pic\Abar$ is any lift of $\lambda$ under the surjective map in~\eqref{eq:NSX} (with $A$ in place of $X$). This may be proved using the theorem of the square \cite[page  59, Corollary 4]{MumfordBook}. We define $\phi_\calN = \phi_\lambda$.

\begin{Lemma}\label{Lemma2}
	For any $x \in X(\Kbar)$ the pullback map $\tau_x^*\colon\Pic\Xbar \to \Pic\Abar$ is an isomorphism of abelian groups which induces isomorphisms of Galois modules $\tau_x^*\colon\Pic^0\Xbar \to \Pic^0\Abar$ and $\tau_x^*\colon \NS\Xbar \to \NS\Abar$.
\end{Lemma}

\begin{proof}
	Since $\tau_x\colon \Abar \to \Xbar$ is an isomorphism of varieties over $\Kbar$ it induces an isomorphism of exact sequences of abelian groups
	\begin{equation}\label{eq:nsseq}
		\xymatrix{
			0 \ar[r]& \Pic^0\Xbar \ar[d]^{\tau_x^*} \ar[r]& \Pic\Xbar \ar[d]^{\tau_x^*} \ar[r]& \NS\Xbar \ar[d]^{\tau_x^*} \ar[r]& 0\\
			0 \ar[r]& \Pic^0\Abar \ar[r]& \Pic\Abar \ar[r]& \NS\Abar \ar[r]& 0\,.
		}
	\end{equation}
	Let $\calL \in \Pic\Xbar$ and $\sigma \in \Gal(K)$. For the unique $b \in A(\Kbar)$ determined by $\mu(b,x) = \sigma(x)$ we have $\tau_{\sigma(x)}^* = t_b^*\circ \tau_x^*$ as maps $\Pic \Xbar \to \Pic \Abar$. Setting $\calM = \tau_x^*(\sigma(\calL))$ we have 
	\begin{align*}\label{eq:diff}
			\sigma(\tau_x^*(\calL))\otimes \tau_x^*(\sigma(\calL))^{-1} &= \tau_{\sigma(x)}^*(\sigma(\calL)) \otimes \tau_x^*(\sigma(\calL))^{-1}\\
			&= t_{b}^*\tau_x^*(\sigma(\calL)) \otimes \tau_x^*(\sigma(\calL))^{-1}\\
			&= t_b^*\calM\otimes\calM^{-1}\\
			&= \phi_\calM(b)\,.
	\end{align*}
	Galois equivariance of the map $\tau_x^* \colon  \NS\Xbar \to \NS\Abar$ follows from the fact that $\phi_\calM(b) \in \Pic^0\Xbar$. Galois equivariance of the map $\tau_x^*\colon \Pic^0\Xbar \to \Pic^0\Abar$ follows from the fact that $\phi_\calM = 0$ when $\calL \in \Pic^0\Xbar$, because then $\calM \in \Pic^0\Abar$ and so its image in $\NS\Abar$ is trivial.
	 \end{proof}
	 
	 \begin{Lemma}\label{lem:deltas}
	 	Let $\delta_A,\delta_X \colon \HH^0(K,\NS\Abar) \to \HH^1(K,\Pic^0\Abar)$ denote the connecting homomorphisms in the Galois cohomology of the rows of~\eqref{eq:nsseq}, where for $\delta_X$ we use the isomorphisms $\tau_x^*$ given by Lemma~\ref{Lemma2} to identify $\NS\Xbar$ and $\Pic^0\Xbar$ with $\NS\Abar$ and $\Pic^0\Abar$ , respectively. Suppose $\lambda \in \HH^0(K,\NS\Abar)$ is such that $\delta_A(\lambda) = 0$. Then $\delta_X(\lambda) = \pm\phi_\lambda([X])$, where $[X]$ denotes the class of $X$ in $\HH^1(K,A)$.
	 \end{Lemma}
	 
	 \begin{proof}
	 	Let $\calN \in \Pic\Abar$ be a lift of $\lambda$ and let $\calL \in \Pic \Xbar$ be such that $\tau_x^*\calL = \calN$. Then $\delta_X(\lambda)$ is the class of the $1$-cocycle given by $\tau_x^*\left(\sigma(\calL)\otimes\calL^{-1}\right)$. Let $b_\sigma \in A(\Kbar)$ be the $1$-cochain uniquely determined by $\mu(b_\sigma,\sigma(x)) = x$. Then $\tau_x^* = t_{b_\sigma}^*\circ \tau_{\sigma(x)}^*$. We compute
	 	\begin{align*}
	 		\tau_x^*\left(\sigma(\calL)\otimes\calL^{-1}\right) &= t^*_{b_\sigma}(\tau_{\sigma(x)}^*(\sigma(\calL))) \otimes \tau_x^*(\calL^{-1}) \\
	 		&= t_{b_\sigma}^*(\sigma(\tau_x^*(\calL)))\otimes \calN^{-1}\\
	 		&= t_{b_\sigma}^*(\sigma(\calN)) \otimes \calN^{-1}\\
	 		&= t_{b_\sigma}^*\calN \otimes \calN^{-1}\\
	 		&= \phi_\lambda(b_\sigma)\,,
	 	\end{align*}
	 	where the penultimate equality follows from the assumption $\delta_A(\lambda) = 0$. To conclude we note that (up to sign) $b_\sigma$ represents the class of $X$ in $\HH^1(K,A)$.
	 \end{proof}

The following result may be found in \cite{Berkovich}. We repeat the argument here for the sake of completeness.

\begin{Lemma}\label{lem:nBr}
	Let $\Abar$ be an abelian variety over an algebraically closed field of characteristic $0$. Then the multiplication-by-$n$ endomorphism of $\Abar$ induces multiplication by $n^2$ on the abelian groups $\NS \Abar$ and $\Br \Abar$ and multiplication by $n$ on $\Pic^0\Abar$.
\end{Lemma}

\begin{proof}
	The statement for $\Pic^0\Abar$ is well known (see \cite[page 75, (iii)]{MumfordBook}). 	For any $m \ge 1$, the Kummer sequence gives an exact sequence,
	\[
		0 \to \NS \Abar/m \to \HH^2_{\text{\'et}}(\Abar,\mu_m) \to \Br \Abar[m] \to 0\,,
	\]
	where we have used the fact that $\Pic \Abar/m \simeq \NS \Abar/m$, since $\Pic \Abar$ is an extension of $\NS \Abar$ by the divisible group $\Pic^0\Abar$. Since $\Br \Abar$ is torsion and $\NS \Abar$ is finitely generated it suffices to show that $[n]$ induces multiplication by $n^2$ on the middle term, $\HH^2_{\text{\'et}}(\Abar,\mu_m)$. From the Kummer sequence one obtains $\HH^1_{\text{\'et}}(\Abar,\mu_m) \simeq (\Pic\Abar)[m] \simeq \Hom(A[m],\mu_m)$. Then there are isomorphisms
\[
	\HH^2_{\text{\'et}}(\Abar,\mu_m)  \simeq \wedge^2 \HH^1_\text{\'et}(\Abar,\mu_m)(-1) \simeq \wedge^2 \Hom(A[m],\mu_m)(-1) \simeq \Hom( \wedge^2A[m],\mu_m)\,,
\]
and on the final term $[n]$ clearly acts as multiplication by $n^2$.
\end{proof}

For an integer $m$, an \defi{$m$-covering of $X$} is an $X$-torsor under $A[m]$ whose \defi{type} (cf. \cite[Section 2.3]{Skorobogatov}) is the map $\lambda_m \colon \Hom(A[m],\Kbar^\times) \to \Pic\Xbar$ obtained by composing the canonical isomorphism $\Hom(A[m],\Kbar^\times) \simeq (\Pic^0\Abar)[m]$ with the inverse of an isomorphism $\tau_x^*\colon\Pic^0\Xbar \simeq \Pic^0\Abar$ given by Lemma~\ref{Lemma2} and the obvious inclusions. See \cite[Section 3.3]{Skorobogatov} for more on $m$-coverings.

	The following is a variant of \cite[Lemma 4.6]{CV17}.

\begin{Lemma}\label{Lemma1}
	There exists an integer $p \ge 1$ (depending only on $A$) such that if $n$ is an integer and $\pi\colon Y \to X$ is an $n^p$-covering of $X$. Then the induced map $\pi^*\colon \Br X/\Br_0X \to \Br Y/\Br_0Y$ annihilates the $n$-torsion.
\end{Lemma}

\begin{Remark}
	It would be interesting to determine if the this lemma holds with $p = 1$. This is the case when $\Br X = \Br_1 X$ and $\NS\Xbar = \Z$, so in particular when $\dim X = 1$.
\end{Remark}

\begin{proof}
	In this proof we abbreviate $\HH^i(K,\bullet)$ to $\HH^i(\bullet)$.
	Consider the exact sequence
    		\[
    			\HH^0(\Pic\Abar) \To \HH^0(\NS\Abar) \stackrel{\delta_A}{\To} \HH^1(\Pic^0\Abar) \stackrel{\epsilon_A}\To \HH^1(\Pic\Abar)
    		\]
    		coming from the Galois cohomology of the bottom row of~\eqref{eq:nsseq}. Since $\NS\Abar$ is finitely generated and $\HH^1(\Pic^0\Abar)$ is torsion, the image of $\delta_A$ is finite. Choose $q\ge 0$ such that for any integer $n$, the $n$-primary subgroup of the image of $\delta_A$ is annihilated by $n^q$. Clearly $q$ depends only on $A$. We will show that $p = q + 3$ will suffice.
    		
    		Suppose $n$ is an integer and $\pi \colon Y \to X$ is an $n^p$-covering. For any $i \ge 0$, multiplication by $n^i$ yields an exact sequence of Galois modules,
	\[
		0 \to A[n^{i}] \To A[n^{i+1}] \stackrel{[n^i]}\To A[n] \to 0\,.
	\]
	In particular, the successive quotients in the filtration $0 \subset A[n] \subset A[n^2] \subset \cdots \subset A[n^p]$ are all isomorphic to $A[n]$. From this we see that $\pi$ factors as
	\begin{equation}\label{eq:pi_is}
		 \pi \colon Y = Y_p \stackrel{\pi_p}\To Y_{p-1} \stackrel{\pi_{p-1}} \To \cdots \stackrel{\pi_2}\To Y_1 \stackrel{\pi_1}\To Y_0 = X\,,
	\end{equation}
	where each $Y_i$ is a torsor under $A = \Alb^0_{Y_i}$ and $\pi_i\colon Y_i \to Y_{i-1}$ is an $n$-covering.
	It may help the reader to note that multiplication by $n^i$ induces a map $n^i_* \colon \HH^1_\text{\'et}(X,A[n^p]) \to \HH^1_\text{\'et}(X,A[n^{p-i}])$ sending the class of $Y$ to the class of $Y_{p-i}$. 

Choose $y_p \in Y(\Kbar)$ and let $y_i$ denote its image in $Y_i(\Kbar)$ under the maps in~\eqref{eq:pi_is}. These $y_i$ induce isomorphisms $\tau_{y_i}\colon \Abar \to \overline{Y_i}$ which, by \cite[Prop. 3.3.2(ii)]{Skorobogatov}, fit into a commutative diagram
	\[
		\xymatrix{
			\Abar \ar[r]^{[n]}  \ar[d]^{\tau_{y_p}}& \Abar \ar[r]^{[n]}\ar[d]^{\tau_{y_{p-1}}}  & \cdots \ar[r]^{[n]} &\Abar \ar[r]^{[n]} \ar[d]^{\tau_{y_1}}  & \Abar\ar[d]^{\tau_{y_0}}\\
			\overline{Y_p} \ar[r]^{\pi_p} & \overline{Y}_{p-1} \ar[r]^{\pi_{p-1}} & \cdots \ar[r]^{\pi_2} &\overline{Y_1} \ar[r]^{\pi_1}& \overline{X}
		}
	\]	
	Moreover, the $\tau_{y_i}$ induce isomorphisms of groups $\tau_{y_i}^*\colon\Br\overline{Y_i} \simeq \Br\Abar$ and, by Lemma~\ref{Lemma2}, isomorphisms of Galois modules $\tau_{y_i}^*\colon\Pic^0\overline{Y_i} \simeq \Pic^0\Abar$ and $\tau_{y_i}^*\colon\NS\overline{Y_i} \simeq \NS\Abar$. Under these identifications the maps $\pi_i^*$ induce, by Lemma~\ref{lem:nBr}, multiplication by $n$, $n^2$ and $n^2$ on the abelian groups $\HH^1(\Pic^0\Abar)$, $\HH^1(\NS\Abar)$ and $\Br\Abar$, respectively.
	
	There are injective maps $r_i\colon \Br_1Y_i/\Br_0Y_i \to \HH^1(\Pic\overline{Y_i})$ (cf. \eqref{eq:definer}) and exact sequences
	\[
		\HH^0(\NS\overline{Y_i}) \to \HH^1(\Pic^0\overline{Y_i}) \to \HH^1(\Pic\overline{Y_i}) \to \HH^1(\NS\overline{Y_i})\,, \text{ and}
	\]
	\[
		0 \to \Br_1Y_i/\Br_0Y_i \to \Br Y_i/\Br_0Y_i \to \Br \overline{Y_i}\,.
	\]
	 Putting this all together, we obtain the following commutative diagram of abelian groups with exact rows. (Both maps from the second row to the bottom row are given by $\pi_p^*\circ\dots\circ\pi_2^*$.)

  		\[
    			\xymatrix{
    				&&&& \frac{\Br X}{\Br_0X} \ar[r] \ar[d]^{\pi_1^*}& \Br \Abar \ar[d]^{n^2} \\
    				&&&\frac{\Br_1Y_1}{\Br_0Y_1}  \ar@{^{(}->}[r] \ar@/_1.5pc/[dl]_{r_1} \ar@/^5.5pc/[dddd]&\frac{\Br Y_1}{\Br_0Y_1}\ar[r] \ar[dddd]^-{\pi_p^*\circ\dots\circ\pi_2^*}&\Br \Abar \\
				&& \HH^1(\Pic \overline{Y_1}) \ar[r] \ar[d]^{\pi_2^*}& \HH^1(\NS\Abar)\ar[d]^{n^2} \\
				\HH^0(\NS\Abar) \ar[r]^{\delta_2}\ar[d]^-{n^2} &\HH^1(\Pic^0\Abar) \ar[r]^{\epsilon_2} \ar[d]^-n& \HH^1(\Pic \overline{Y_2}) \ar[r]\ar[d]^-{\pi_3^*}& \HH^1(\NS\Abar) \\
				\vdots\ar[d]^-{n^2} & \vdots \ar[d]^-n & \vdots \ar[d]^-{\pi_p^*}\\
    				\HH^0(\NS\Abar) \ar[r]^{\delta_p} & \HH^1(\Pic^0\Abar) \ar[r]^{\epsilon_p}& \HH^1(\Pic \overline{Y_p}) & \frac{\Br_1Y_p}{\Br_0Y_p} \ar@{_{(}->}[l]_-{r_p}\ar@{^{(}->}[r]& \frac{\Br Y_p}{\Br_0Y_p}
    			}
    		\]
    		
    		Now suppose $\alpha \in \Br X/\Br_0X$ has order dividing $n$. Exactness of the second row shows that $\pi_1^*(\alpha) = \alpha_1$ for some $\alpha_1 \in \Br_1Y_1/\Br_0Y_1$. Let $\alpha_2 = (\pi_2^*\circ r_1)(\alpha_1)$ be the image of $\alpha_1$ in $\HH^1(\Pic\overline{Y_2})$. We will show that $(\pi_p^*\circ \cdots \circ \pi_3^*)(\alpha_2) = 0$ in $\HH^1(\Pic\overline{Y_p})$. Commutativity of the diagram then gives $r_p(\pi^*(\alpha)) = 0$. But $r_p$ is injective, so this gives $\pi^*(\alpha) = 0$ as required.
    		
    		Since $n\alpha_1 = 0$, exactness of the fourth row shows that $\alpha_2 = \epsilon_2(\beta)$ for some $\beta \in \HH^1(\Pic^0\Abar)$. Since $\HH^1(\Pic^0\Abar)$ is torsion, we may assume $\beta$ is annihilated by some power of $n$. Since $n\alpha_2 = 0$, we have $n\beta \in \ker(\epsilon_2)$. Hence $n\beta = \delta_2(\gamma)$ for some $\gamma \in \HH^0(\NS\Abar)$. Our choice of $q$ ensures that $\delta_A(n^q\gamma ) = 0$. Hence, by Lemma~\ref{lem:deltas} (applied with $Y_i$ in place of $X$), we have $\delta_i(n^q\gamma) = \pm\phi_{n^q\gamma}([Y_i])$, for $i = 0,\dots,p$. Since $n^j_*[Y_i] = [Y_{i-j}]$ in $\HH^1(\Pic^0\Abar)$ this gives (up to sign)
    		\[
    			n^{q+1}\beta = n^q\delta_2(\gamma) = \delta_2(n^q\gamma) = n^j\delta_{2+j}(n^q\gamma) \in \ker(\epsilon_{2+j}), \quad \text{for $j \ge 0$}\,.
    		\]
    		Applying this in the case $j = p-2 = q+1$ and using commutativity of the diagram we have
    		\[
    			0 = \epsilon_p(n^{q+1}\beta) = (\pi_p^*\circ \cdots \circ \pi_3^*\circ \epsilon_2)(\beta) = (\pi_p^*\circ \cdots \circ \pi_3^*)(\alpha_2)\,.
		\]
\end{proof}

The proof of Theorem~\ref{thm:completelycaptures} requires the following elementary result in linear algebra.

\begin{Lemma}\label{lem:linearalgebra}
	Let $M \in \Mat_{s,t}(\Z/n\Z)$ be an $s\times t$ matrix with rows ${\bf m}_i$ and let ${\bf c} \in \Mat_{s,1}(\Z/n\Z)$ be a column vector with entries $c_i$. The equation $M{\bf x} - {\bf c} = {\bf 0}$ has no solution with ${\bf x} \in \Mat_{t,1}(\Z/n\Z)$ if and only if there is some $\Z/n\Z$-linear combination of the row vectors $[{\bf m}_i | c_i ]$ of the form $[ {\bf 0} | c]$ with $c \ne 0$.
\end{Lemma}

\begin{proof}
        If such a linear combination exists, then clearly there is no
solution. For the converse, we use that there are invertible matrices
$P \in \GL_s(\Z/n\Z)$, $Q \in \GL_t(\Z/n\Z)$ and a diagonal matrix $D
= \operatorname{diag}(d_i) \in \operatorname{Mat}_{s,t}(\Z/n\Z)$ with
$d_i \mid d_{i+1}$ for $i = 1, \dots, \min\{s,t\}-1$ such that $PM =
DQ$. This follows from the existence of a Smith Normal Form of a lift
of $M$ to $\operatorname{Mat}_{s,t}(\Z)$, since reduction modulo $n$
gives a homomorphism $\GL_r(\Z) \to \GL_r(\Z/n\Z)$. For notational
convenience set $d_i := 0$ for $i > t$ and ${\bf b} := P{\bf c}$. As
$P$ is invertible, there is a solution to $M{\bf x} = {\bf c}$ if and
only if there is a solution to $DQ{\bf x} = {\bf b}$. As $Q$ is
invertible, such a solution exists if and only if $d_iy_i = b_i$ has a
solution with $y_i \in \Z/n\Z$ for all $i= 1, \dots, s$. For given
$i$, the equation $d_iy_i = b_i$ can be solved if and only if the
order of $b_i$ divides the order of $d_i$. Therefore, if $M{\bf x} =
{\bf c}$ has no solution, then there is some $i_0$ is such that
$\ord(b_{i_0}) \nmid \ord(d_{i_0})$. Then $\ord(d_{i_0})$ times the
$i_0$-th row of the augmented matrix $[DQ|{\bf b}] = [PM|P{\bf c}]$ is
of the form $[{\bf 0}|c]$ with $c \ne 0$.
        \end{proof}

\section{Compatibility of the Tate and Brauer-Manin pairings}

In this section we assume $X$ is a torsor under an abelian variety $A = \Alb^0_X$ over a local field $K$ of characteristic $0$. Let $\CH^0X$ denote the Chow group of $0$-cycles on $X$ modulo rational equivalence, and let $A^0X$ denote the kernel of the degree map on $\CH^0X$. For a class $\alpha \in \Br X$ the evaluation map $\alpha \colon X(K) \to \Br K$ induces a homomorphism $\alpha \colon \CH^0X \to \Br K$ defined by sending the class $z \in \CH^0X$ of a closed point $P$ to the image under the corestriction map $\Cor_{K(P)/K}\colon \Br K(P) \to \Br K$ of $\alpha(P) \in \Br K(P)$. This gives a bilinear pairing on $\CH^0X \times \Br X$. When restricted to $\Br_{1/2}X$, this pairing is compatible with the Tate pairing in the sense that there is a commutative diagram of pairings,

	\begin{equation}\label{eq:pairings}
		\begin{array}[c]{cccccc}
			\langle\,,\,\rangle_\textup{Tate}\colon &\Alb^0_X(K) & \times & \HH^1\left(K,\Pic^0 \Xbar\right) & \rightarrow & \Br K\\
			&\rotatebox{90}{$\to$} && \rotatebox{90}{$\leftarrow$}&& \rotatebox{90}{$=$}\\
			&A^0X & \times & \HH^1(K,\Pic\Xbar) & \rightarrow & \Br K\\
			&\rotatebox{90}{$\hookleftarrow$} && \rotatebox{90}{$\rightarrow$}&& \rotatebox{90}{$=$}\\
			\langle\,,\,\rangle_\textup{eval}\colon &\CH^0X& \times & \Br_1 X & \rightarrow & \Br K\,.
		\end{array}
	\end{equation}
	The maps to $\HH^1(K,\Pic\Xbar)$ in the diagram are the maps $i_*$ and $r$ defined in Section~\ref{notation}. Since $K$ is a local field, $r$ is surjective. The pairing in the middle row is induced by $\langle \,,\, \rangle_\textup{eval}$. This is well defined because $\ker(r) = \Br_0X$ which pairs trivially with $A^0X$. For $X$ a genus $1$ curve, this goes back to Lichtenbaum \cite{Lichtenbaum}. The general case is given by \cite[Proposition 3.4]{Kai}. The reader will note that while the statement there assumes the existence of a proper regular model for $X$, the proof that the pairings are compatible does not make use of this assumption. Alternatively, when $X(K) \ne \emptyset$ (which is the only case we will use) one can apply \cite[Theorem 3.5]{CipKrash}.

The compatibility of the pairings in~\eqref{eq:pairings} has the following consequences.

\begin{Corollary}\label{cor:2}
	For any $\alpha \in \Br_{1/2}X$, the homomorphism $\alpha \colon A^0X \to \Br K$ factors through $\Alb^0_X(K)$. Moreover, for any $a \in \Alb^0_X(K)$ and $x \in X(K)$ we have $\alpha(\mu(a,x)) = \alpha(x) + \alpha(a)$.
\end{Corollary}

\begin{Corollary}\label{cor:1}
	Suppose $X(K) \ne \emptyset$ and that $\alpha \in \Br_{1/2}X$ and $\alpha' \in \HH^1(K,\Pic^0\Xbar)$ are such that $i_*(\alpha') = r(\alpha)$. The following are equivalent:
	\begin{enumerate}
		\item $\alpha' = 0$;
		\item $\alpha \in \Br_0X$;
		\item The evaluation map $\alpha\colon X(K) \to \Br K$ is constant.
	\end{enumerate}
\end{Corollary}

\begin{proof}
	We trivially have (1) $\Rightarrow$ (2) $\Rightarrow$ (3), since $\ker(r) = \Br_0X$. It thus suffices to show that (3) implies (1). For this, suppose the evaluation map is constant and let $x_0 \in X(K)$. The Albanese map sending $x \in X(K)$ to the class of $x-x_0$ surjects onto $\Alb^0_X(K)$. The compatibility in~\eqref{eq:pairings} together with the fact that $\alpha$ is constant on $X(K)$ therefore implies that $\alpha'$ is in the right kernel of the Tate pairing. But the right kernel is trivial, so $\alpha' = 0$.
\end{proof}

\section{Proofs of the theorems}\label{sec:proofs}

Let $X$ be a torsor under an abelian variety $A$ over a number field $k$ and let $\hat{A} = \Pic^0_X$ be the Picard variety of $X$. Let $\lambda_n \colon \hat{A}[n] \to \Pic \Xbar$ be the type map of an $n$-covering of $X$ and set $\Br_{n} X := r^{-1}(\lambda_{n*}(\HH^1(k,\hat{A}[n])))$ where $r$ is as in~\eqref{eq:definer}.

\begin{Proposition}\label{prop:desc}
	Suppose $X$ is locally soluble. The following are equivalent.
	\begin{enumerate}
		\item $X(\A)^{\Br_nX}= \emptyset$.
		\item No adelic point on $X$ lifts to an $n$-covering of $X$.
		\item The class of $X$ in $\Sha(k,A)$ is not divisible by $n$.
		\item There exists $Y$ in $\Sha(k,\hat{A})[n]$ pairing nontrivially with $X$ under the Cassels-Tate pairing.
	\end{enumerate}
	These equivalent conditions imply that there exists $\beta \in \Bcyr X[n]$ such that $X(\A)^\beta = \emptyset$. Moreover, if $X(\A)^{\Br_{1/2}} = \emptyset$, then there exists an $n$ for which the conditions above hold.
\end{Proposition}

\begin{proof}
	(1) and (2) are equivalent by descent theory (e.g., \cite[Theorem 6.1.2(a)]{Skorobogatov}). The equivalence of (2) and (3) follows from \cite[Proposition 3.3.5]{Skorobogatov}. When $X$ is divisible by $n$ in $\HH^1(k,A)$, the equivalence of (3) and (4) follows from \cite[Lemma I.6.17]{MilneADT}, while when $X$ is not divisible by $n$ in $\HH^1(k,A)$ it follows from \cite[Theorem 4]{CreutzBLMS}. Assuming (4), let $\beta \in \Bcyr X$ be such that its image in $\HH^1(k,\Pic\Xbar)$ is equal to that of $Y$. Modifying by a constant algebra if necessary we may take $\beta \in \Bcyr X[n]$. That $X(\A)^\beta = \emptyset$ for such $\beta$ follows immediately from \cite[Th\'eor\`eme 6]{Manin}. 
	
	For the final statement we use that $\Br_{1/2}X = \bigcup_{n \ge 1} \Br_{n}X$, which holds because $\HH^1(k,\hat{A})$ is torsion and the maps $\HH^1(k,\hat{A}[n]) \to \HH^1(k,\hat{A})[n]$ are all surjective. Noting also that $\Br_dX \subset \Br_nX$ for $d \mid n$, a compactness argument shows that if $X(\A)^{\Br_{1/2}}= \emptyset$, then $X(\A)^{\Br_nX} = \emptyset$ for some $n$.
	\end{proof}

\begin{proof}[Proof of Theorem~\ref{thm:completelycaptures}]
	Let $B \subset \Br_{1/2} X$ be a subgroup such that $X(\A_k)^B = \emptyset$. If $X(\A_k) = \emptyset$, then $\beta = 0$ satisfies the conclusion of the theorem. Hence we may assume that $X(\A_k) \ne \emptyset$. By compactness we may assume that $B$ is finitely generated, say $B = \langle \alpha_1,\dots,\alpha_m \rangle$. Let $n$ be such that the $n\alpha_i = 0$ for all $i = 1,\dots,m$, which exists since $\Br X$ is a torsion group by a result of Grothendieck. Let $(y_v) \in X(\A_k)$ and define $c_i := \langle (y_v),\alpha_i\rangle_\textup{BM} = \sum_v \inv_v \alpha_i(y_v) \in \Q/\Z[n]$. Let $S$ be a finite set of places of $k$ such that the evaluation maps $\alpha_i \colon X(k_v) \to \Br k_v$ are trivial for all $i = 1,\dots,m$ and all $v \notin S$. Set $V := \prod_{v \in S}\Alb^0_X(k_v)/n$. By Corollary~\ref{cor:2} each $\alpha_i$ induces a homomorphism of finite $\Z/n\Z$-modules $\phi_i \colon V \to \Q/\Z[n]$ sending $(a_v)_{v\in S}$ to $\sum_{v \in S}\operatorname{inv}_v\alpha_i(a_v)$. Moreover, by the second statement of the corollary, the condition $X(\A_k)^B = \emptyset$ is equivalent to the insolubility of the system of $\Z/n\Z$-linear equations $\phi_i({\bf x}) = c_i$. By Lemma~\ref{lem:linearalgebra} there are $b_i \in \Z/n\Z$ such that $\sum{b_i\phi_i} = 0$ with $\sum {b_i}c_i = c \ne 0$. Then $\beta := \sum {b_i\alpha_i} \in B$ takes constant value $c \ne 0$ on $X(\A_k)$. Hence $X(\A_k)^\beta = \emptyset$. Moreover, the maps $\beta \colon X(k_v) \to \Br k_v$ are necessarily all constant, which implies that $\beta \in \Bcyr X$ by Corollary~\ref{cor:1}.
	\end{proof}

\begin{proof}[Proof of Theorem~\ref{prop2}]
	This follows immediately from Theorem~\ref{thm:completelycaptures} or (independently) from Proposition~\ref{prop:desc}. 
\end{proof}

\begin{proof}[Proof of Theorem~\ref{prop1}]
	It suffices to prove $X(\A_k)^{\Br_{1/2}} \subset X(\A_k)^{\Br}$, the other containment being obvious. Let $x \in X(\A_k)^{\Br_{1/2}}$ and let $\alpha \in \Br X$. Since $\Br X$ is torsion, there is an integer $n \ge 1$ such that $n\alpha = 0$. By Proposition~\ref{prop:desc} the assumption $x \in X(\A_k)^{\Br_{1/2}}$ implies that $x$ lifts to an adelic point $y \in Y(\A_k)$ for some $n^p$-covering $\pi\colon Y \to X$, where $p$ is as given by Lemma~\ref{Lemma1}. Now Lemma~\ref{Lemma1} and functoriality of the Brauer-Manin pairing give $\langle x,\alpha \rangle_\textup{BM} = \langle \pi(y) , \alpha \rangle_\textup{BM} = \langle y,\pi^*\alpha \rangle_\textup{BM} = 0\,,$ showing that $x \in X(\A_k)^\alpha$. Since $\alpha$ was arbitrary, $x \in X(\A_k)^{\Br}$.
\end{proof}

\begin{proof}[Proof of Theorem~\ref{thm:mainthm}]
	This is an immediate consequence of Theorems~\ref{prop1} and~\ref{prop2}.
\end{proof}

\begin{proof}[Proof of Theorem~\ref{cor1}]
	By Theorem~\ref{thm:mainthm}, the Brauer-Manin obstruction to the existence of rational points on a torsor $X$ under an abelian variety $A$ is equivalent to the obstruction coming from $\Bcyr X$. By \cite[Th\'eor\`eme 6]{Manin} and \cite[Theorem I.6.13(a)]{MilneADT} there is an obstruction coming from $\Bcyr X$ if and only if $X$ is not divisible in $\Sha(k,A)$. We conclude by noting that the subgroup of divisible elements in $\Sha(k,A)$ coincides with the maximal divisible subgroup of $\Sha(k,A)$, because $\Sha(k,A)$ is torsion and each $n$-torsion subgroup is finite (cf.,  \cite[Remark I.6.7]{MilneADT}).
\end{proof}

\begin{proof}[Proof of Theorem~\ref{cor2}]
	Suppose that $X(\A_k)^B = \emptyset$ with $B \subset \Br X[n]$. Lemma~\ref{Lemma1} implies that no adelic point of $X$ lifts to an $n^p$-covering. Indeed, if $\pi\colon Y \to X$ is an $n^p$-covering with $Y(\A_k) \ne \emptyset$, then using functoriality of the Brauer pairing as in the proof of Theorem~\ref{prop1} we see that the points of $\pi(Y(\A_k))$ are orthogonal to $B$. The theorem then follows from Proposition~\ref{prop:desc}.
\end{proof}

\subsection*{Acknowledgements} The author thanks David Harari, Bianca Viray and Olivier Wittenberg for helpful comments and discussions regarding this paper. The author was partially supported by the Marsden Fund Council administered by the Royal Society of New Zealand.


\end{document}